%% file: main.tex
\documentclass[a4paper,reqno]{amsart}
\usepackage[british]{babel}
\numberwithin{equation}{section}
\usepackage{amssymb}
\usepackage[protrusion=true,expansion=true]{microtype}	
\usepackage{amsmath,amsfonts,amsthm} % Math packages
\usepackage{amsmath,amssymb,amsthm, amsfonts,color}
\usepackage[pdftex]{graphicx}	
\usepackage{url}
\usepackage[title,titletoc,page]{appendix} 
\usepackage{listings} 
\usepackage{color}
\usepackage{comment}
\usepackage{caption}
\usepackage{subcaption}
\usepackage{algorithm}

\usepackage[backend=biber,style=numeric]{biblatex}
\usepackage{fancyhdr}
\numberwithin{equation}{section}		% Equationnumbering: section.eq#
\numberwithin{figure}{section}			% Figurenumbering: section.fig#
\numberwithin{table}{section}				% Tablenumbering: section.tab#

\usepackage[left=3.5cm,right=3.5cm,top=3.5cm,bottom=3.5cm,footskip=1.5cm,headsep=1cm]{geometry}
\usepackage[%
pdftitle={Notes for EP sensing},
pdfkeywords={abc},
]{hyperref}
\input{commands}

\title[Exceptional point enhanced detection]{Exceptional point enhanced small particle detection in systems of subwavelength resonators}
\author{Jinghao Cao\thankssymb{1} and Jörg Nick\thankssymb{2}} 
\thanks{\thankssymb{1} Computing and Mathematical Sciences Department, California Institute of Technology, 1200 East California Boulevard, Pasadena, CA 91125, USA, E-mail: jinghao.cao@caltech.edu} 
\thanks{\thankssymb{2} Seminar für Angewandte Mathematik, ETH Zürich, Rämistrasse 101,  CH-8092 Zürich, Switzerland. E-mail: joerg.nick@math.ethz.ch}
\begin{document}
\begin{abstract}
This paper considers the effects of small highly contrasted particles on the subwavelength resonances of a system of high-contrast resonators, with an application to sensing. The key technique is a multiple scattering expansion of the capacitance matrix associated with the perturbed system. At leading order, the perturbation of the scattering resonances is characterized by the associated term of the truncated multiple scattering expansion. When an exceptional point is present in the resonance structure, the perturbation critically affects the subwavelength resonances, which improves the sensitivity of a sensing task in the presence of noise. Numerical experiments demonstrate  the use of the proposed reconstruction techniques.
\end{abstract}
\maketitle
\section{Introduction}
A subwavelength resonator is a highly contrasted bounded 
inclusion that exhibits subwavelength resonances. Subwavelength resonances are frequencies at which the resonator strongly interacts with incident waves whose wavelengths can be larger by several orders of magnitude \cite{ammari.davies.ea2021Functional}. When systems of subwavelength resonators  operate at an exceptional point, their behaviour can be greatly influenced, even by very small perturbations. 
As a result, such systems are ideal building blocks for designing small and powerful sensors capable of detecting the presence of small particles such as viruses and nanoparticles. They can measure the shifts in the resonant frequencies of the structure caused by perturbations \cite{rechtsman2017applied, vollmer2008whispering, vollmer2008single} even though the noise is proportional to the size of the perturbation. 
	
An exceptional point is a point in parameter space at which two or more eigenvalues, and also the corresponding eigenvectors, coincide  \cite{heiss2012physics, miri2019exceptional}. It arises in non-Hermitian systems of subwavelength resonators by introducing gain and loss, which can be represented by the imaginary parts of complex-valued material parameters \cite{ammari.davies.ea2022Exceptional, ADHHY21}.

\section{The forward problem}
In this paper we consider a chain of high-contrast material inclusions called subwavelength resonators. More precisely, we let $$\mathcal D = \mathcal{D}_1\cup \dots \cup \mathcal{D}_N \subset \mathbb{R}^3,$$ be a union of simply connected, bounded domains of class $C^{1,s} $ with $0<s<1$. 
%Then the one-dimensional array is given by $\mathcal{D} = \bigcup_{i=1}^N \mathcal{D}_i$. 
Consider the following set of Helmholtz equations, of the unperturbed system of resonators:
\begin{equation}
    \label{helmholtz}
    \begin{cases}
    \ds \Delta u + k^2u = 0 \ & \text{in } \R^3 \setminus \overline{\mathcal{D}}, \\
    \nm
    \ds \Delta u + k_{\mathcal{D}}^2 u = 0 \ & \text{in } \mathcal{D}, \\
    \nm
    \ds  u|_{+} -u|_{-}  = 0  & \text{on } \p \mathcal{D} \text{ and } \p \Omega, \\
    \nm
    \ds \left.\delta \frac{\p u}{\p \nu} \right|_{+} - \left.\frac{\p u}{\p \nu} \right|_{-} = 0 & \text{on } \p \mathcal{D} \text{ and } \p \Omega, \\
    \nm
    \ds u \text{ satisfies an outgoing radiation condition}. &   
    \end{cases}
\end{equation}
 % \begin{equation}
 %     \label{helmholtz}
 %    \begin{cases}
 %     \ds \Delta u + k_m^2u = 0 \ & \text{in } \R^3 \setminus \overline{\mathcal{D}}, \\
 %     \nm
 %     \ds \Delta u + k_r^2 u = 0 \ & \text{in } \mathcal{D}, \\
 %     \nm
 %     \ds  u|_{+} -u|_{-}  = 0  & \text{on } \p \mathcal{D} \text{ and } \p \Omega, \\
 %     \nm
 %     \ds \left.\delta \frac{\p u}{\p \nu} \right|_{+} - \left.\frac{\p u}{\p \nu} \right|_{-} = 0 & \text{on } \p \mathcal{D} \text{ and } \p \Omega, \\
 %     \nm
 %     \ds u \text{ satisfies an outgoing radiation condition}. &   
 %     \end{cases}
 % \end{equation}
Here $\nu$ is the outgoing normal to $\partial \mathcal{D}$ and we denote the limits from outside and inside of $\mathcal{D}$ by $|_+$ and $|_-$. The inclusions $\mathcal{D}$ and the surrounding materials $\mathbb{R}^3\setminus\mathcal{D}$ carry material parameters denoted by $\rho_r,\kappa_r$ and $\rho_m,\kappa_m$. In acoustic metamaterials these are referred to as the density and the bulk modulus. In the Helmholtz equations, we also define the following variables
\begin{equation}
    v_{\mathrm{m}}=\sqrt{\frac{\kappa_{\mathrm{m}}}{\rho_{\mathrm{m}}}}, \quad v_{\mathcal{D}}=\sqrt{\frac{\kappa_{\mathcal{D}}}{\rho_{\mathrm{r}}}}, \quad k =\frac{\omega}{v_{\mathrm{m}}} \quad \text { and } \quad k_{\mathcal{D}}=\frac{\omega}{v_{\mathcal{D}}}.
\end{equation}
We call $\omega\in\mathbb{C}$ a resonant frequency if there exists an associated nontrivial solution to \eqref{helmholtz}. 
In this paper, we consider the problem of adding one resonator $\Omega$ to the system
\begin{equation}
        \tilde{\mathcal{D}} = \mathcal D \cup \Omega,
\end{equation}
with a different material parameter, resulting in a different wave speed $v_\Omega = \sqrt{\tfrac{\kappa_\Omega}{\rho_\Omega}}$. 
The additional inclusion with high contrast is sometimes referred to by the index $N+1$, in the sense that we implicitly identify $\Omega=\mathcal D_{N+1}$, when it is convenient (which mainly appears in the context of the capacitance matrix). The perturbed coupled system of Helmholtz equations then reads
\begin{equation}
    \label{adaptedhelmholtz}
    \begin{cases}
    \ds \Delta u + k^2u = 0 \ & \text{in } \R^3 \setminus \overline{\mathcal{D}}, \\
    \nm
    \ds \Delta u + k_{\mathcal{D}}^2 u = 0 \ & \text{in } \mathcal{D}, \\
    \nm
    \ds \Delta u + k_\Omega ^2 u = 0 \ & \text{in } \Omega, \\
    \nm
    \ds  u|_{+} -u|_{-}  = 0  & \text{on } \p \mathcal{D} \text{ and } \p \Omega, \\
    \nm
    \ds \left.\delta \frac{\p u}{\p \nu} \right|_{+} - \left.\frac{\p u}{\p \nu} \right|_{-} = 0 & \text{on } \p \mathcal{D} \text{ and } \p \Omega, \\
    \nm
    \ds u \text{ satisfies an outgoing radiation condition}. &   
    \end{cases}
\end{equation}
When the jump conditions of \eqref{helmholtz} are only imposed on $\partial \mathcal D$, then we refer to this boundary value problem as the \emph{unperturbed} system.

\subsection{Potential operators and capacitance matrices}
Throughout the manuscript, we use boundary integral equations to characterize solutions of the high-contrast limit of the Helmholtz system, in the high-contrast limit. When $\delta\rightarrow 0$ and $k\rightarrow 0$ in an appropriate sense, the problem \eqref{helmholtz} approaches the Poisson problem in a well-understood sense \cite{ammari2024functional}. This is known as the \emph{subwavelength regime}.   
We therefore introduce the boundary integral operators of the Poisson problem, where we denote the domain by an additional index, namely we write
$$S_{ \mathcal{D}}\, \colon \, H^{-1/2}(\partial \mathcal{D}) \rightarrow H^{1/2}(\partial \mathcal{D}), $$
for the single-layer boundary operator, which reads
\begin{align*}
S_{\mathcal{D}}(\varphi) 
= 
\int_{\partial \mathcal{D}}\frac{1}{4\pi |x-y|} \varphi(y) \mathrm d y.
\end{align*}
% We assume that the original domain  consists of $N$ separated inclusions, which are denoted by $$\mathcal \mathcal{D} = \cup_{j=1}^N \mathcal{\mathcal{D}}_j .$$ 
The exterior domain is referred to as $\mathcal D^+ = \mathbb R^{3}\setminus \mathcal D$.
Different characterizations of the entries of the capacitance matrix can be formulated, (see e.g. \cite{ammari.davies.ea2021Functional,cbms}). Here, we rely on the boundary integral equation formulation, which reads for $i, j=1, \ldots, N$ 
\begin{align}\label{eq:def-C}
C_{ij} = \int_{\partial \mathcal{D}_i}  S_\mathcal{D}^{-1}[\chi_{\partial \mathcal{D}_j}] (x)
\mathrm d x .
\end{align}
The resonance frequencies of the exterior Helmholtz system of resonators $\mathcal D$, which do not impose the jump conditions along the boundary of $\mathcal D$ are determined, up to defects of order $\mathcal{O}(\delta^{1/2})$, by the eigenvalues of the capacitance matrix
\begin{equation}
    \mathcal{C}_{i j}=\frac{\delta_i v_i^2}{\left|\mathcal{D}_i\right|} C_{i j}, \quad i, j=1, \ldots, N .
\end{equation}
%Here, we use the volume integral formulation 
%\begin{equation}\label{eq:unperturbed-capacitance}
% C_{i j} = \int_{\mathcal D^+}\nabla V_i \cdot \nabla V_j \, \mathrm d x,   \quad i, j=1, \ldots, N,
%\end{equation}
%where the harmonic functions $V_i\in H^{1}(\mathcal D^+)$ for $i=1,\dots,N$ fulfill the boundary value problems
%\begin{alignat}{3}\label{eq:def-Vi-1}
%\Delta V_i &= 0 \quad &&\text{in} \quad&&\mathcal D^+, \\ 
%V_i &= \delta_{ij} \quad &&\text{on}\quad \partial &&\mathcal D_j.\label{eq:def-Vi-2}
%\end{alignat}
%The systems are completed by installing the appropriate asymptotic conditions $V_i = \mathcal{O}(\frac{1}{|x|})$ for $|x|\rightarrow \infty$.

The following result gives a characterization of the resonance frequencies of the system \eqref{helmholtz}, via an explicit relation to the eigenvalues of the capacitance matrix $\mathcal C$ \cite{ammari2024functional}.
\begin{lemma}\label{lem:capacitance}
 Consider a system of $N$ subwavelength resonators in $\mathbb{R}^3$. As $\delta \rightarrow 0$, the $N$ subwavelength resonant frequencies satisfy the asymptotic formula
\begin{equation}
\omega_n=\sqrt{\lambda_n}+O(\delta), \quad n=1, \ldots, N .
\end{equation}
\end{lemma}
Throughout the paper, we assume that the defect $\Omega$ has some distance away from the resonator
\begin{align}\label{eq:def-dist}
    \dd = \mathrm{dist}(\mathcal D ,  \Omega)\ge c_\dd >0.
\end{align}
We track the dependency on $\dd$ for defects that are sufficiently large, in the sense that all constants in estimates may depend on $\dd$ only through the lower bound $c_\dd$.
The capacitance matrix is then determined again by the expression \eqref{eq:def-C}, now for $ i, j=1, \ldots, N+1$, which reads 
%Again, we formulate the capacitance matrix of the system, which we denote by
%\begin{equation}\label{eq:unperturbed-capacitance2}
% \tilde C_{i j} = \int_{\mathcal D^+}\nabla \tilde V_i \cdot \nabla \tilde V_j \, \mathrm d x,   \quad i,j=1,\ldots, N+1, 
%\end{equation}
%where the functions $\tilde V_i\in H^{1}(\tilde{\mathcal D}^+)$ for $i=1,\dots,N$ fulfill the corresponding boundary value problems
%\begin{alignat}{3}\label{eq:def-Vi-tilde-1}
%\Delta \tilde V_i &= 0 \quad &&\text{in} \quad&&\tilde{\mathcal D}^+, \\ 
%\tilde{V}_i &= \delta_{ij} \quad &&\text{on}\quad \partial &&\mathcal D_j \quad \text{ for } \ j = 1,\dots,N+1.\label{eq:def-Vi-tilde-2}
%\end{alignat}
%Again, we further impose the usual asymptotic conditions for $|x|\rightarrow \infty $, in order to make the solution to \eqref{eq:def-Vi-tilde-1}--\eqref{eq:def-Vi-tilde-2} unique.
\begin{align*}
\widetilde{C}_{ij} = \int_{\partial \mathcal{D}_i}  S_{\widetilde{\mathcal{D}}}^{-1}[\chi_{\partial\mathcal{D}_j}] (x)
\mathrm d x .
\end{align*}
Note that the index $N+1$ refers to the small particle $\Omega$.
The resonant structure of the perturbed system, at leading order, is now again captured by Lemma~\ref{lem:capacitance}. The objective of this note is to derive an expansion of this perturbation of the capacitance matrix, which is explicit with respect to the magnitude of the small particle $\Omega$. In fact, we assume to work in a physical regime, such that 
\begin{align}\label{eq:def-regime}
 \dfrac{|\partial \Omega|^{1/2}}{\dd} \ll 1. 
\end{align}
The additional resonator $\Omega$ is thus assumed to be either very small or far away from the resonating structure. Measuring shifts of resonance frequencies can be done with great accuracy, in many practical situations. The effect of the resonating defect particle $\Omega$ on the resonating structure, which is captured at leading order by the capacitance matrix, is therefore of primary concern in many sensing applications. 

\begin{remark}[Assumptions on the contrast of the perturbation]
   Throughout the manuscript, we assume that the perturbation $\Omega$ is itself a highly contrasted particle, which covers e.g. the setting bubbles in a fluid. Moreover, the analysis covers the perturbations of additional highly contrasted resonators far away from the resonating structure. For particles $\Omega$, which are not highly contrasted, techniques based on an additional expansion (see  \cite[Section~3.2]{AFKRYZ18}) with respect to the magnitude of the perturbation may provide effective tools to extend the present theory to particles consisting of materials that are not highly contrasted.
\end{remark}

\section{The effect of resonating perturbations on the capacitance matrix}
We start with an analysis of the effect of the additional particle on the entries of the capacitance matrix.
%We are especially iwhose magnitude $|\\mathcal D_{N+1}|\ll 1$ is assumed to be significant smaller than the size of the resonators $D_1,\dots,D_N$. Moreover, we assume that the pertubation is separated from the original resonators by some distance $d=\text{dist}(D,D_{N+1})$. We are interested in the effect of this pertubation on the capacitance matrix and its spectral properties, with a particular focus on understanding the consequences of the small magnitude and the distance of the pertubation.  

% \begin{figure}[h]
%     \centering
%     \includegraphics[width = \textwidth]{pert_in_d.eps}
%     \caption{Perturbation in distances}
%     \label{fig:enter-label}
% \end{figure}

% \begin{figure}[h]
%     \centering
%     \includegraphics[width = \textwidth]{pert_in_r.eps}
%     \caption{Perturbation in radius}
%     \label{fig:enter-label}
% \end{figure}

\subsection{Expansion of the defects in the capacitance matrix}
We write the single-layer boundary operator on $\widetilde{\mathcal{D}}$ as the block form
\begin{align}\label{eq:SDtilde}
    S_{\widetilde{\mathcal{D}}} =
    \begin{pmatrix}
        S_{\mathcal{D}} & S_{\mathcal{D},\Omega} \\
        S_{\Omega, \mathcal{D}} & S_\Omega
    \end{pmatrix},
\end{align}
which decomposes $$S_{\widetilde{\mathcal{D}}}\, \colon \, H^{-1/2}(\partial \mathcal{D}) \times  H^{-1/2}(\partial \Omega)\rightarrow H^{1/2}(\partial \mathcal{D}) \times  H^{1/2}(\partial \Omega), $$
into its respective components. We have the following bounds, which are explicit with respect to the surface of the defect $\Omega$.
\begin{lemma}\label{lem:components-SD}
Suppose that $|\partial \mathcal{D}| = \mathcal{O}(1)$,  $|\partial \Omega| \ll 1$ and $1=\mathcal{O}(1)$ in the sense of \eqref{eq:def-dist}. 
Then, the components of the block operator \eqref{eq:SDtilde} fulfill the bounds
\begin{alignat*}{2}
&\| S_\mathcal{D}^{-1} \|_{H^{-1/2}(\partial \mathcal{D})\leftarrow H^{1/2}(\partial \mathcal{D})}&&\le c,
\\
&\| S_\Omega^{-1} \|_{H^{-1/2}(\partial \Omega)\leftarrow H^{1/2}(\partial \Omega)}&&\le c |\partial \Omega |^{-1/2}.
\end{alignat*}
Moreover, we have the following estimates for the mixed operators on the anti-diagonal
\begin{alignat*}{2}
&\| S_{\mathcal{D},\Omega} \|_{H^{1/2}(\partial \mathcal{D})\leftarrow H^{-1/2}(\partial \Omega)}&&\le c \frac{|\partial \Omega |^{1/2}}{\dd},
\\
&\| S_{\Omega, \mathcal{D}} \|_{H^{1/2}(\partial \Omega)\leftarrow H^{-1/2}(\partial \Omega)}&&\le c \frac{|\partial \Omega |^{1/2}}{\dd} .
\end{alignat*}
The constant $c$ is independent of $|\Omega|$, but depends on the geometry of the resonators and the defect $\Omega$, as well as the constant $c_\dd$.
\end{lemma}
\begin{proof}
The bounds on the inverses of $S_\mathcal{D}^{-1}$ and $S_\Omega^{-1}$ follow from standard bounds and a rescaling argument. 
For the mixed operator $S_{\mathcal{D},\Omega}$, we have the pointwise estimate, for $x\not \in \partial \Omega$
\begin{align*}
| \nabla_x S_{\mathcal{D},\Omega}(\varphi_\Omega) |
= 
\left|\int_{\partial \Omega} \nabla_x \frac{1}{4\pi |x-y|} \varphi_\Omega(y) \mathrm d y \right|
 &\le
\left\| \nabla_x \frac{1}{4\pi |x-y|} \right\|_{H^{1/2}(\partial\Omega)} \left \| \varphi_\Omega\right \|_{H^{-1/2}(\partial\Omega)}.
%\\ & \le
%C |\partial \Omega |^{1/2} \left \| \varphi_\Omega\right \|_{H^{-1/2}(\partial\Omega)}.
\end{align*}
Integrating over $\partial \mathcal{D}$ yields the estimate 
\begin{align*}
 \| S_{\mathcal{D},\Omega}(\varphi_\Omega) \|_{H^{1/2}(\partial \mathcal{D})}
 & \le c \frac{|\partial \Omega |^{\tfrac{1}{2}}}{\dd} \left \| \varphi_\Omega\right \|_{H^{-1/2}(\partial\Omega)}.
\end{align*}
Conversely, we use the same estimate chain to obtain
\begin{align*}
 \| S_{\Omega, \mathcal{D}}(\varphi_\mathcal{D}) \|_{H^{1/2}(\partial \Omega)}
 & \le c\frac{|\partial \Omega |^{\tfrac{1}{2}}}{\dd}\left \| \varphi_\mathcal{D}\right \|_{H^{-1/2}(\partial \mathcal{D})}.
\end{align*}
\end{proof}
We define the following auxiliary operators, which correspond to two reflections of the resonators $\mathcal{D}$ and the defect $\Omega$ respectively:
\begin{equation}\label{eq:def-reflections}
    T_\mathcal{D} = S_\mathcal{\mathcal{D}}^{-1} S_{\mathcal{D},\Omega}S_{\Omega}^{-1}S_{\Omega, \mathcal{D}} \quad \text{and} \quad 
T_\Omega = S_{\Omega}^{-1}S_{\Omega, \mathcal{D}} S_{\mathcal{D}}^{-1} S_{\mathcal{D}, \Omega}.
\end{equation}
These operators are bounded by the following lemma.
\begin{lemma}\label{lem:reflections}
The operators corresponding to reflections fulfill the bound
  \begin{align*}
  \left\| T_\mathcal{D}  \right\|_{H^{-1/2}(\partial \mathcal{D}) \leftarrow H^{-1/2}(\partial \mathcal{D})}  
&\le 
c\frac{|\partial \Omega |^{\tfrac{1}{2}}}{\dd},
\\
\left\|T_\Omega\right\|_{H^{-1/2}(\partial \Omega) \leftarrow H^{-1/2}(\partial \Omega) } 
    &\le c\frac{|\partial \Omega |^{\tfrac{1}{2}}}{\dd}.
\end{align*}
The constant $C$ in both estimates is independent of $|\partial \Omega|$, but in general depends on the geometry of all resonators.
\end{lemma}

We are in a position to formulate a Neumann series expansion of the inverse of $S_{\widetilde{\mathcal{D}}}$, which formally reads 
%\begin{align}\label{eq:Neumann-series}
%    S_{\widetilde D}^{-1}
%    &=
%    \sum_{j=0}^{\infty}(-1)^j
%     \left(\begin{pmatrix}
%        S_\mathcal{D}^{-1} & 0 \\ 
%        0        & S_{\Omega}^{-1}
%    \end{pmatrix}   
%\begin{pmatrix}
%    0 & S_{D\Omega} \\
%    S_{\Omega D} & 0 
%\end{pmatrix}
%\right)^j\begin{pmatrix}
%        S_\mathcal{D}^{-1} & 0 \\ 
%        0        & S_{\Omega}^{-1}
%    \end{pmatrix}.
%\end{align}
\begin{align}\label{eq:Neumann-series}
    S_{\widetilde{\mathcal{D}}}^{-1}
    &=
    \sum_{j=0}^{\infty}(-1)^j 
\begin{pmatrix}
    0 & S_{\mathcal{D}}^{-1} S_{\mathcal{D}, \Omega} \\
    S_{\Omega}^{-1}S_{\Omega, \mathcal{D}} & 0 
\end{pmatrix}^j
    \begin{pmatrix}
        S_{\mathcal{D}}^{-1} & 0 \\ 
        0        & S_{\Omega}^{-1}
    \end{pmatrix}.
\end{align}
We note that the operator norm of the component $S_{\Omega}^{-1}S_{\Omega, \mathcal{D}}$ generally does not tend to zero as $|\partial \Omega |\rightarrow 0$.
This difficulty can be overcome by considering two reflections at once, with the help of the auxiliary operators defined in \eqref{eq:def-reflections}. The following lemma discusses the convergence properties of the Neumann series.
\begin{prop}
For $|\partial\Omega|$ sufficiently small, or $d$ sufficiently large, the series \eqref{eq:Neumann-series} is convergent. Moreover, the sequence of partial sums $P_{2K}$ (which contains the first $2K-1$ summands of \eqref{eq:Neumann-series}) fulfills the bound 
\begin{align*}
\left\| S_{\widetilde{\mathcal{D}}}^{-1} - P_{2K} \right\|_{H^{-1/2}(\partial \widetilde{\mathcal{D}})\leftarrow H^{1/2}(\partial \widetilde{ \mathcal{D}})} \le c \left(\dfrac{|\partial \Omega|^{\tfrac{1}{2}}}{d}\right)^{K-1}.
\end{align*}
The constant $C$ is independent of $|\partial \Omega|$, but may depend on the geometry of all resonators.
\end{prop}
\begin{proof}
The series is derived by using 
\begin{align*}
    S_{\widetilde{\mathcal{D}}}^{-1} 
    &= 
     \begin{pmatrix}
        S_\mathcal{D} & S_{\mathcal{D},\Omega} \\
        S_{\Omega, \mathcal{D}} & S_\Omega\end{pmatrix}^{^{-1}}.
\end{align*}
We start with the convergence of the series and observe that 
\begin{align*}
\begin{pmatrix}
    0 & S_\mathcal{D}^{-1} S_{\mathcal{D}, \Omega} \\
    S_{\Omega}^{-1}S_{\Omega, \mathcal{D}} & 0 
\end{pmatrix}^2 
& = 
\begin{pmatrix}
    T_\mathcal{D} & 0  \\
    0 & T_\Omega
\end{pmatrix},
\end{align*}
where the auxiliary operators $T_\mathcal{D}$ and $T_\Omega$ are defined in \eqref{eq:def-reflections} and bounded by Lemma~\ref{lem:reflections}.

Separating odd and even powers of the power series yields
%\begin{align}\label{eq:odd-and-even}
%    \sum_{n=0}^{\infty}\left( I - \begin{pmatrix}
%    0 & S_\mathcal{D}^{-1} S_{D\Omega} \\
%    S_{\Omega}^{-1}S_{\Omega D} & 0 
%\end{pmatrix}    \right)
%\begin{pmatrix}
%    0 & S_\mathcal{D}^{-1} S_{D\Omega} \\
%    S_{\Omega}^{-1}S_{\Omega D} & 0 
%\end{pmatrix}^{2n}
%    \begin{pmatrix}
%        S_\mathcal{D}^{-1} & 0 \\ 
%        0        & S_{\Omega}^{-1}
%    \end{pmatrix},
%\end{align}
%or more specifically 
\begin{align}
     S_{\widetilde{\mathcal{D}}}^{-1} 
     &=
     \begin{pmatrix}
    I & - S_\mathcal{D}^{-1} S_{\mathcal{D},\Omega} \\
    -S_{\Omega}^{-1}S_{\Omega, \mathcal{D}} & I 
\end{pmatrix}
\sum_{n=0}^{\infty}
\begin{pmatrix}
    T_\mathcal{D}^nS_\mathcal{D}^{-1} & 0  \\
    0 & T_\Omega^nS_{\Omega}^{-1}
\end{pmatrix} \nonumber
\\ &
= 
\sum_{n=0}^{\infty}
\begin{pmatrix}
    T_\mathcal{D}^nS_\mathcal{D}^{-1} & - S_\mathcal{D}^{-1} S_{\mathcal{D}, \Omega}T_\Omega^nS_{\Omega}^{-1}   \\
       -S_{\Omega}^{-1}S_{\Omega, \mathcal{D}}T_\mathcal{D}^nS_\mathcal{D}^{-1} & T_\Omega^nS_{\Omega}^{-1}
\end{pmatrix}\label{eq:expansion-block-2}
\end{align}
The operator norms of the block operators appearing in the summands on the right-hand side are bounded by $|\partial \Omega|^{(n-1)/2}$, which gives the stated result.
\end{proof}
Inserting this result into the definition of the capacitance matrix then yields the following expansion for each matrix entry.
\begin{thm}\label{thm:expansion}
The defect of the entries of the capacitance matrix fulfill the expansions
\begin{alignat*}{2}
   \widetilde{C}_{ij} - C_{ij} &= \sum_{n=1}^{\infty}\int_{\partial \mathcal{D}_i}  
T_\mathcal{D}^nS_\mathcal{D}^{-1}[\chi_{\partial\mathcal{D}_j}]
\mathrm d x \quad &&\text{ for } \quad 1\le i,j\le N. 
\end{alignat*}
The additional entries of the capacitance matrix further fulfill the expansions
\begin{alignat*}{2}
 \widetilde C_{i ,N+1} &=  
\sum_{n=0}^{\infty}\int_{\partial \mathcal{D}_i}  
  - S_\mathcal{D}^{-1} S_{\mathcal{D}, \Omega}T_\Omega^nS_{\Omega}^{-1}[\chi_{\partial\mathcal{D}_{N+1}}]
\, \mathrm d x \quad &&\text{ for }\quad 1\le i\le N,
\\
 \widetilde C_{N+1 ,j} &=  
\sum_{n=0}^{\infty}\int_{\partial \Omega}  
  -S_{\Omega}^{-1}S_{\Omega, \mathcal{D}}T_\mathcal{D}^n S_\mathcal{D}^{-1}[\chi_{\partial\mathcal{D}_j}]
\, \mathrm d x \quad &&\text{ for }\quad 1\le j\le N.
\end{alignat*}
Finally, the last remaining additional entry of $\widetilde C$ fulfills 
\begin{align*}
   \widetilde C_{N+1,N+1} 
   = \sum_{n=0}^{\infty}\int_{\mathcal{D}_{N+1}}  
T_\Omega^n S_\Omega^{-1}[\chi_{\partial\mathcal{D}_{N+1}}]
\, \mathrm d x.
\end{align*}

\end{thm}
\begin{proof}
The statement is a direct consequence of \eqref{eq:expansion-block-2}, which yields, for $1\le i,j\le N$
\begin{align*}
\widetilde{C}_{ij} - C_{ij}
&=
\int_{\partial \mathcal{D}_i}    \left(  S_{\widetilde{\mathcal{D}}}^{-1}
-S_\mathcal{D}^{-1}\right)[\chi_{\partial \mathcal{D}_j}] \mathrm d x
= 
\sum_{n=1}^{\infty}\int_{\partial \mathcal{D}_i}  
T_\mathcal{D}^nS_\mathcal{D}^{-1}[\chi_{\partial \mathcal{D}_j}]
\mathrm d x.
\end{align*}
The other identities follow in the same way.
\end{proof}

\begin{cor}
\label{cor:pert}
    There exists a constant $c>0$, depending only on the geometry of the resonators, such that the  entries of the capacitance matrix fulfill    
    \begin{alignat*}{2}
    \left|\widetilde{C} - \begin{pmatrix}
        C & 0 \\ 0 & 0 
    \end{pmatrix} \right| 
&\le c \dfrac{|\partial \Omega |^{\tfrac{1}{2}}}{\dd}.
    \end{alignat*}
    Moreover, collecting an additional term of the expansion yields
    \begin{align*}
        \left|\widetilde{C} - \begin{pmatrix}
        C & 0 \\ 0 & 0 
    \end{pmatrix}-\begin{pmatrix}
        E_{11} & E_{12} \\ E_{21} & E_{22} 
    \end{pmatrix} \right| 
&\le c \dfrac{|\partial \Omega |}{\dd^2},  
    \end{align*}
    where 
    \begin{align*}
      E = 
      %\begin{pmatrix}
      %      E_{11} & E_{12} \\
      %      E_{21} & E_{22}
      %\end{pmatrix}
      %  =
        \begin{pmatrix}
        \left(\int_{\partial \mathcal{D}_i}  
T_\mathcal{D}S_\mathcal{D}^{-1}[\chi_{\partial \mathcal{D}_j}]
\mathrm d x\right)_{i,j=1}^N &    -\left(\int_{\partial \mathcal{D}_i}  
   S_\mathcal{D}^{-1} S_{\mathcal{D},\Omega}T_\Omega S_{\Omega}^{-1}[\chi_{\partial \mathcal{D}_{N+1}}]
\, \mathrm d x \right)_{i=1}^N
\\
 \left(-\int_{\partial \Omega}  
  S_{\Omega}^{-1}S_{\Omega, \mathcal{D}}T_\mathcal{D} S_\mathcal{D}^{-1}[\chi_{\partial \mathcal{D}_j}]
\, \mathrm d x  \right)^N_{j=1}
&
\int_{\mathcal{D}_{N+1}}  
T_\Omega S_\Omega^{-1}[\chi_{\partial \mathcal{D}_{N+1}}]
\, \mathrm d x
        \end{pmatrix} . 
    \end{align*}
\end{cor}
\begin{proof}
The statement follows by estimating the leading order term of the expansions of Theorem~\ref{thm:expansion}, via a dual argument and the given bounds of Lemmas~\ref{lem:components-SD}--\ref{lem:reflections}.
\end{proof}

\section{Perturbation of the resonances}
In the following, we leverage the bounds on the perturbation of the capacitance matrix with standard eigenvalue perturbation arguments for matrices.
Here, we use standard perturbation results found in \cite{L66} (and later reformulated and extended in \cite{MBO97}). A particular focus here is laid on the mathematical analysis of the capacitance matrix at resonance frequencies near exceptional points.

In the following, we let $\omega\in \{ \omega_1,\dots,\omega_N\}$ denote an eigenvalue of the unperturbed capacitance matrix $\mathcal{C}$. Moreover, we let 
$\tilde{\omega} \in \{\tilde{\omega}_1,\dots,\tilde{\omega}_N\}$ denote the corresponding eigenvalue of the corresponding perturbed capacitance matrix $\tilde{\mathcal C}$. The most common case is the case of simple eigenvalue, which is covered in the following lemma.
\begin{lemma}[Simple eigenvalues]\label{lem:simple-resonances}
Let $\omega$ be a simple eigenvalue of the capacitance matrix $\mathcal C$, with right- and left-eigenvectors $x$ and $y$ in $\mathbb R^N$ respectively. Then, the perturbed subwavelength resonance has the form 
\begin{align}\label{eq:expansion-simple-eigenvalue}
    \tilde{\omega} =\sqrt{\omega^2
    + \dfrac{y^{*}  D_\delta E x}{ y^{*} x}
    + \mathcal{O}\left(\dfrac{|\partial \Omega|}{d^2}\right)}.
\end{align}
We note that the simplicity of the eigenvalue $\omega$ here implies that the condition of the eigenvalue $y^*x\neq 0$ is bounded away from zero and can enter into constants of the remainder term.
Inserting the bounds on the perturbation $E$ further yields
\begin{align*}
  |   \tilde{\omega} -\omega| \le C    \dfrac{|\partial \Omega|^{1/2}}{d} .
\end{align*}
The constant $C$ also depends on $\omega$.
\end{lemma}
% \jinghao{I think in the exp case, the formula should look like this}
% Let $k_j$ denote the multiplicity,
% \begin{align}\label{eq:}
%     \tilde{\omega}_{j,i} =\omega_j 
%     + \alpha |\partial \Omega|^{-(2k)}
%     + \mathcal O(|\partial \Omega|).
% \end{align}
Any practical system is technically covered by this result, since arbitrary small defects in the physical parameters or the geometries of the resonators would destroy an exact exceptional point structure. In those cases, the condition of the eigenvalues $y^*x$ is so small that the constants scaling with their inverse can not be neglected for a useful expansion following along the lines of \eqref{eq:expansion-simple-eigenvalue}. Therefore, the idealized setting of defective capacitance matrices can still provide insights by quantifying the amplification of the perturbation \eqref{eq:expansion-simple-eigenvalue} through the exceptional point present in the system.
We therefore turn to the case where the unperturbed capacitance matrix $\mathcal{C}$ is a defective matrix and $\omega$ is an exceptional point of order $k$ at $\omega$. For the sake of presentation, we further assume that $\omega$ has geometric multiplicity $g=1$ (and refer the reader to \cite{MBO97} for a detailed expansion of other settings).
The right Jordan chain with $k$ elements is generated at this point, whose elements $x_1,\dots,x_k$ are generated by $x_k$ via
\begin{align*}
x_{k-j} &= (\mathcal{C} - \omega I)^{j} x_{k}, 
\end{align*}
where $x_1$ is an ordinary eigenvector of the eigenvalue $\omega$. In the same way, we write $y_1,\dots,y_k$ for the associated left Jordan chain of the same eigenvalue. We assume the following normalization 
\begin{align*}
    Y^* X = I, \quad \text{where}\quad X  = [x_1,\dots,x_k], \quad Y^* = [
        y_k, \dots , 
        y_1].
\end{align*}
Then, we have the following expansion of the resonances at exceptional points, which shows that the effects of small defects on the resonances are greatly enhanced when the capacitance matrix has an exceptional point.
\begin{thm}[Exceptional points]\label{prop:exc-expansion}
Let $\omega$ be an eigenvalue of the capacitance matrix $\mathcal C$, with right- and left-eigenvectors $x$ and $y$ in $\mathbb R^N$ respectively, with algebraic multiplicity $r$ and geometric multiplicity $1$. Then, we have the expansion
\begin{align*}
    \tilde{\omega} = \sqrt{ \omega^2 + \xi^{1/r} + \mathcal{O}\Big(|\partial \Omega|^{\tfrac{1}{r}}d^{-\tfrac{2}{r}}\Big)},
\end{align*}
with 
$$\xi = y^* D_\delta E x \quad \text{and} \quad |\xi| = \mathcal O \Big(|\partial \Omega|^{\tfrac{1}{2r}}d^{-\tfrac{1}{r}}\Big).$$  
\end{thm}
\begin{proof}
The statement follows from the combination of the multiple scattering expansion \eqref{thm:expansion}, the construction of $E$ and the perturbation theory of eigenvalues of matrices \cite{L66} (see e.g. \cite[Theorem~2.1]{MBO97}).
\end{proof}
\begin{remark}[Other expansions]
 Natural extensions of Proposition~\ref{prop:exc-expansion} are available from analyzing more complicated block Jordan structures (see \cite{MBO97}) or including more terms of the multiple scattering expansion of Theorem~\ref{thm:expansion}.   
\end{remark}

%\joerg{This, however, is in contradiction to \eqref{eq:expansion-simple-eigenvalue}. It would be in accordance with the simple case, if we assume that the left and right eigenvectors are normalized in the sense that $x^*y =1 $. I could not find the corresponding normalization in \cite{MBO97}.}
%\jinghao{I think the normalization is encoded in (2.2) in the paper}
%, that is, there exists invertible matrix $V$ such that 
%\begin{equation}
%    P^TVD_\delta C V^{-1}P = P^TJP = \begin{pmatrix}
%        \omega_j & 1 \\ 0 & \omega_j
%    \end{pmatrix},
%\end{equation}
%where $J$ is a Jordan matrix with an order $k$ Jordan block. From (\ref{cor:pert}) we know that $\tilde{\mathcal{C}} = \mathcal{C} + \overline{\mathcal{C}} $ with $\bnorm{\overline{\mathcal{C}}} = \mathcal{O}(|\partial \Omega |^{1/2})$ with $\bnorm{}$ being the operator norm. Hence, the eigenvalues of $\tilde{\mathcal{C}}$ correspond to these of $J+V^{-1}\overline{\mathcal{C}}V$. \href{https://www.sciencedirect.com/science/article/pii/002437959290263A}{this paper} tells us the eigenvalues are perturbed to the following order:
%\begin{equation}
%    \tilde{w}_j = \omega_j + \mathcal{O}((|\partial \Omega |^{1/2})^{1/k}).
%\end{equation}
%Here is another reference making the coefficient of the perturbation explicit:
%\url{https://cs.nyu.edu/~overton/papers/pdffiles/lidskii.pdf} Theorem 2.1.

\section{Numerical experiments}
We consider a toy problem of reconstructing the location and magnitude of a (spherical) small particle $\Omega$ from shifts in the subwavelength resonant frequencies $\omega_1,\dots,\omega_N$. 
A simple approach to this problem reads: Let  $\omega^{\text{mes}}_1,\dots,\omega^{\text{mes}}_N$ be the measured subwavelength resonances and further let $\ell$ by
\begin{align}\label{eq:loss}
\ell(p) := \sum_{j=1}^N \left|\omega_j^{\text{mes}} - \omega_j (p)  \right|^2,
\end{align}
where $\omega_j (p)$ is the associated perturbed subwavelength frequencies and the perturbation is determined by the parameters collected in $p$. 

Consider a chain of resonators, with three spheres of radii $\tfrac{1}{3}$ and the centers $$z_j = (j-1,0,0)^T\in \mathbb R^d\quad  \text{for} \quad j=1,2,3.$$ We study the effects of an additional small perturbation at the center $z = (3,0,0)^T$, in the following two experiments.
\subsubsection*{Effect of perturbations as $|\Omega|\rightarrow 0$}
First, we consider a sequence of decreasing radii and consider the effect on the spectrum of the capacitance matrix. The perturbation affects subwavelength resonances at exceptional points substantially stronger, as is observed in Figure~\ref{fig:pert-r}. Here, physical parameters that correspond to a system with an exceptional point have been constructed via \cite{ADHHY21}.
In particular, we observe the predicted rates of Corollary~\ref{cor:pert}.
\begin{figure}
    \centering  \includegraphics[width=0.6\linewidth]{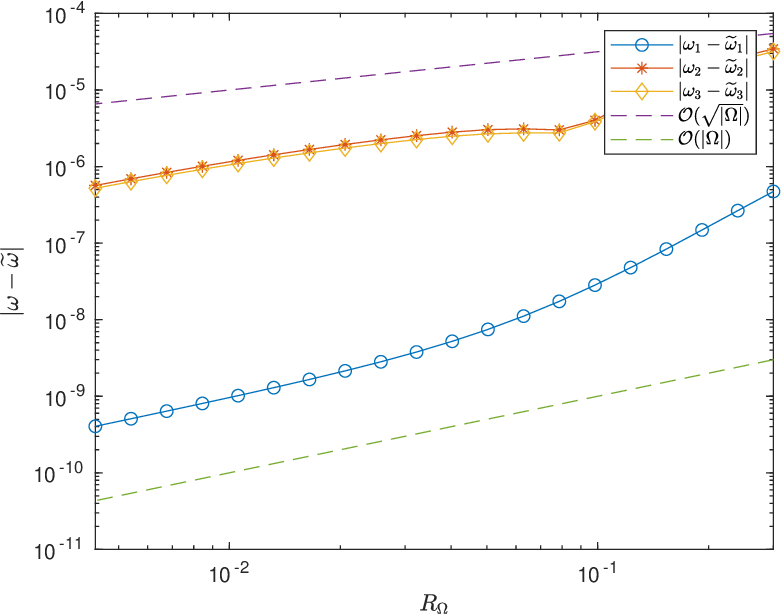}
    \caption{Perturbation of the simple subwavelength resonance ($\omega_1$) and the degenerate resonance at $\omega_2$ and $\omega_3$, when an additional highly contrasted resonator with a decreasing radius is present.}
    \label{fig:pert-r}
\end{figure}
\subsubsection*{The idealized loss $\ell$} Although the expression \eqref{eq:loss} is not derived analytically, however, we can still quickly evaluate it using the capacitance matrix approximation. (In our experiments, we used structures composed of spheres, where the spatial discretization is conducted with spherical harmonics). We approximate the spectrum of the system with radius $10^{-4}$. The position of the defect now varies along the $z$-axis. We note that the setting is radially symmetric and therefore the $z=0$ plane with $y \ge 0$ fully determines the idealized loss function. Here, $\ell_*\approx 0$ corresponds to the centers of perturbations (of magnitude $10^{-4}$), which would cause the observed shifts in the resonant frequencies.

 Figure~\ref{fig:loss} visualizes the logarithm of the loss. We observe that a path of centers opens in the $z=0$ plane, which explain the shifts in resonances, which by the radial symmetry of the setup implies a $2-$ dimensional subset of $\mathbb R^3$, which would explain the measurements observed. Note that the exact perturbation is at $z_\Omega = (3,0,0)^T$, with a radius of $r_\Omega =10^{-4}$. We note that the loss increases substantially when an exceptional point is present, as seen in the right plot.
 \begin{figure}
   \hspace{-0cm}\includegraphics[width=1.1\linewidth,trim=30mm 0mm 0mm 0mm,clip]{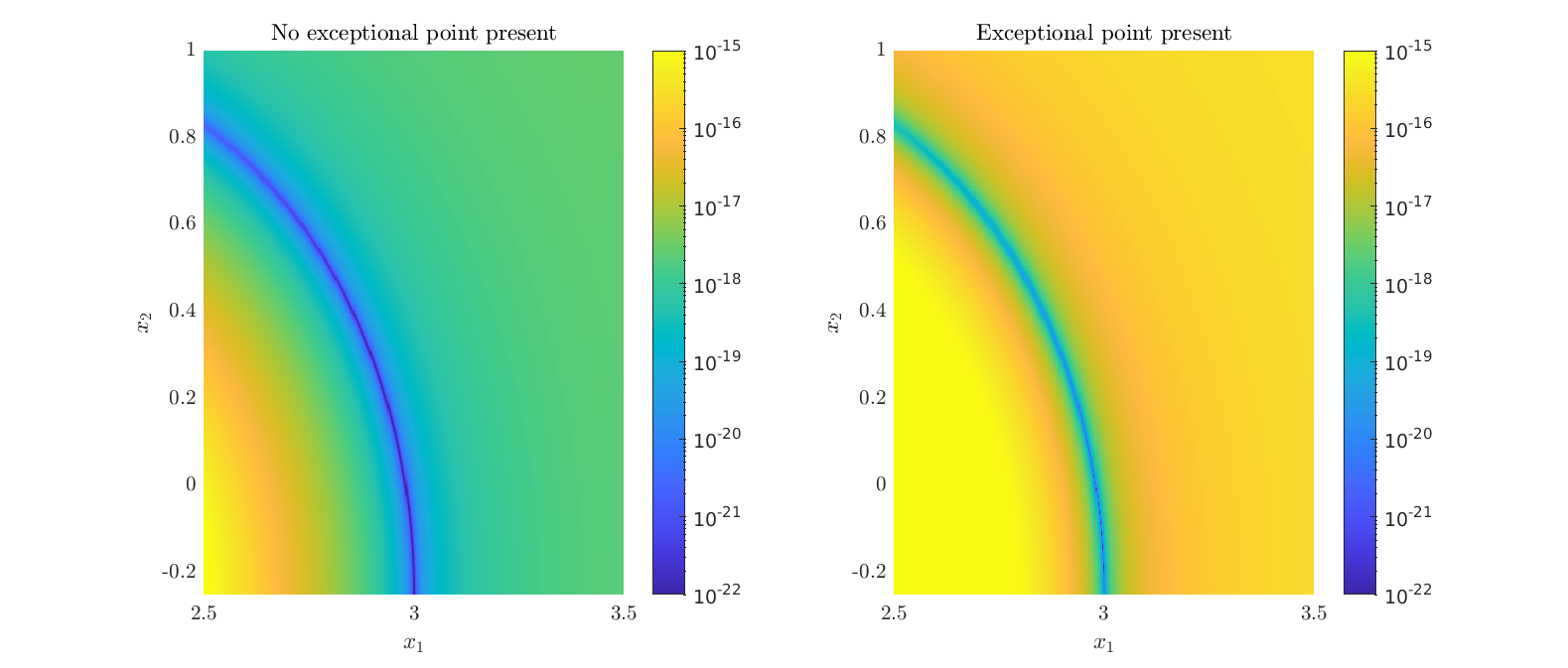}
     \caption{Visualization of the  loss $\ell$, visualized on the $x_3=0$-plane, for a fixed radius $R_\Omega=10^{-4}$. The exact center is set at $z = (3,0,0)^T$.}
     \label{fig:loss}
 \end{figure}

\subsubsection*{A sensing application} 
We are now interested in the effects of an exceptional point on basic sensing applications.
The radius of the spherical perturbation with the fixed and known an radius $R_\Omega = 10^{-2}$. 
The center at $z_\Omega = (3,0,0)$ is to be reconstructed from the subwavelength resonators.   As a simple sensing scheme, we use a steepest descent method applied to the loss shown in Figure \eqref{fig:loss}. 
This steepest descent method then reads
\begin{align}\label{eq:descent-method}
  p_{k+1} = p_k - \lambda^k  \nabla_p \ell_*(p).
\end{align}
We approximate $\nabla_p \ell$ by the central difference quotient. We use the radial symmetry to further eliminate a degree of freedom in the perturbation, by restricting the computation to the $z=0$ plane. The results are depicted in Figure~\ref{fig:steepest-descent}, where the red dots denote some initial guesses $z_0\in [2.5,3.5]\times [0,1]$ and the green dots visualize the iterations after $20$ iterations of the 
For our experiment, we used $\lambda = \tfrac{9}{10}$. The resulting method projects the initial guesses onto the surface of plausible perturbations, that is observed from observing the graphs in Figure~\ref{fig:loss}. In that sense, the methodology could be thought of as a building block of a larger method, that either uses multiple structures or additional information on the perturbations from further measurements. Moreover, using a more complicated structure of resonators, which in particular break symmetries, could provide settings where more information on the perturbation can be drawn from the shifts in the resonance frequencies.

We simulate noise by sampling from a uniform distribution $\eta\in [-\epsilon_{\text{noise}} ,\epsilon_{\text{noise}} ]$, for various maximal levels of noise $\epsilon_{\text{noise}}\ll 1$ and then use \eqref{eq:descent-method} with the loss function applied to noisy measurements, i.e.,  
\begin{align}\label{eq:idealized-loss-noisy2}
\ell^\eta(p) = \sum_{j=1}^N \left|(1+\eta_j )\omega_j^{\text{mes}} - \omega_j (p)  \right|^2.
\end{align}
We apply this methodology to a set of initial nodes for various levels of noise and different initial nodes. To visualize the effects of this noise, we sample $\eta$ and draw $100$ evaluations, for each of which we apply the method \eqref{eq:descent-method} with the noisy measurements as described in \eqref{eq:idealized-loss-noisy2}. The results are reported in Figure~\ref{fig:steepest-descent}, which presents the results for both resonating structures that contain exceptional points and those that do not. Then, when the noise levels increase, the existence of an exceptional point significantly enhances the signal-to-noise ratio and consequently allows the reconstruction of the small particle from measurements of the shift in the subwavelength resonances of the system.
\begin{remark}

When an exceptional point is present in the spectrum of the capacitance matrix, the loss \eqref{eq:loss} and subsequently the descent method \eqref{eq:descent-method} 
 are dominated by changes in the exceptional points. Using a weighted loss
\begin{align}\label{eq:idealized-loss2}
\ell^\alpha(p) := \sum_{j=1}^N \alpha_j\left|\omega_j^{\text{mes}} - \omega_j (p)  \right|^2,
\end{align}
can be used to balance the magnitude of these terms, which may increase the amount of information extracted from the shifts in the spectrum. In our experiments with three resonators along a chain, no substantial improvements were observed. Moreover, we note that large factors $\alpha_j\gg 1$ amplify noise to a substantial degree, which must be balanced with any advantage of such a scaling.
\end{remark}

\begin{figure}
    \centering  \includegraphics[width=0.9\linewidth]{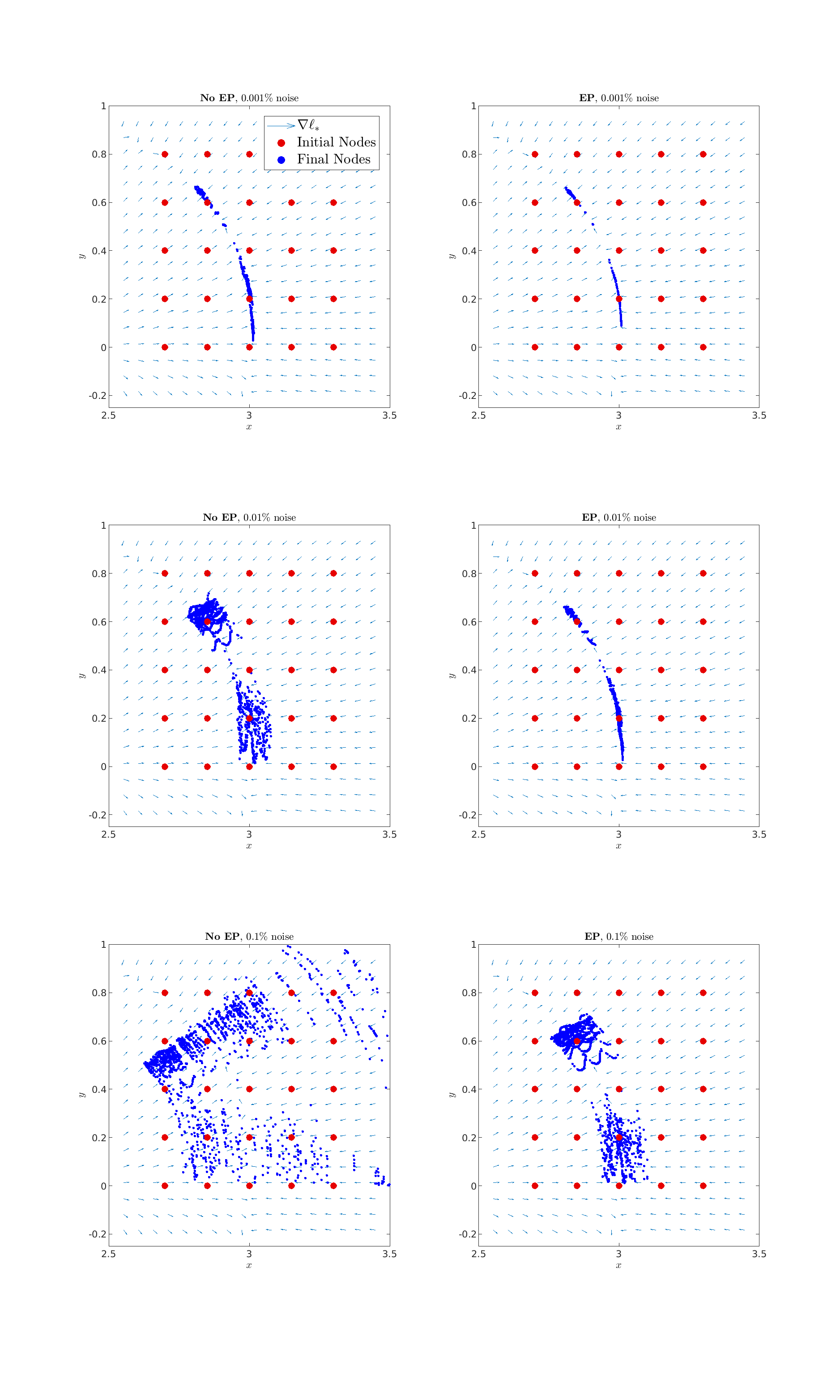}
    \caption{Applying the steepest descent method \eqref{eq:descent-method} to a set of nodes near the exact perturbation (sitting at $z=(3,0,0)$). On the right plots, the physical parameters have been tuned, such that the capacitance matrix has an exceptional point. From the top down an increasing level of noise is simulated in the measurements.}
    \label{fig:steepest-descent}
\end{figure}
\section{Acknowledgment}
The authors thank Habib Ammari for his helpful comments during the writing of the manuscript.
\appendix
% \newpage
\printbibliography
\end{document}

%% file: commands.tex
\usepackage[T1]{fontenc} % To switch to the T1 encoding
\usepackage{lmodern} % To switch to Latin Modern
\rmfamily % To load Latin Modern Roman and enable the following NFSS declarations.
\DeclareFontShape{T1}{lmr}{b}{sc}{<->ssub*cmr/bx/sc}{}
\DeclareFontShape{T1}{lmr}{bx}{sc}{<->ssub*cmr/bx/sc}{}
\usepackage{lipsum}
\usepackage{mathtools}
\usepackage{amsmath}
\usepackage{amssymb}
\usepackage{tikz}
\usetikzlibrary{matrix,positioning,decorations.pathreplacing,calc}
\usetikzlibrary{
decorations.text,%
decorations.markings,%
shadows}
\usepackage{adjustbox}
\usepackage{graphicx} 

\usepackage{caption}
\usepackage{subcaption}

\usepackage{dsfont} % for 1 bbmath
\usepackage[format=hang]{caption}

\usepackage[normalem]{ulem}

\usepackage{bm}

\DeclareNameAlias{default}{family-given}

\graphicspath{{figures/}}

\AtEveryBibitem{% Clean up the bibtex rather than editing it
 \clearfield{url}
 \clearfield{issn}
 \clearfield{isbn}
 \clearfield{urldate}
 
 \ifentrytype{book}{
  \clearfield{pages}}{% 
 }
}

\usepackage{xargs}                      % Use more than one optional parameter in a new commands
\usepackage[prependcaption,textsize=tiny,textwidth=3cm]{todonotes}
\newcommandx{\unsure}[2][1=]{\todo[linecolor=red,backgroundcolor=red!25,bordercolor=red,#1]{#2}}
\newcommandx{\change}[2][1=]{\todo[linecolor=blue,backgroundcolor=blue!25,bordercolor=blue,#1]{#2}}
\newcommandx{\info}[2][1=]{\todo[linecolor=OliveGreen,backgroundcolor=OliveGreen!25,bordercolor=OliveGreen,#1]{#2}}
\newcommandx{\improvement}[2][1=]{\todo[linecolor=black,backgroundcolor=black!25,bordercolor=black,#1]{#2}}
\newcommandx{\thiswillnotshow}[2][1=]{\todo[disable,#1]{#2}}
\usepackage{amsthm}
\usepackage[noabbrev,capitalize]{cleveref}
\crefname{proposition}{Proposition}{Propositions}
\crefname{equation}{}{}

\newtheorem{thm}{Theorem}[section]
\newtheorem{cor}[thm]{Corollary}
\newtheorem{lemma}[thm]{Lemma}
\newtheorem{remark}[thm]{Remark}

\newtheorem{prop}[thm]{Proposition}

\crefname{assumption}{Assumption}{Assumptions}
\crefname{definition}{Definition}{Definitions}
\crefname{corollary}{Corollary}{Corollaries}
\crefname{enumi}{item}{items}
\usepackage{fancyhdr}

\fancypagestyle{plain}{%
\fancyhead[C]{}
\fancyfoot[C]{}

}
\fancypagestyle{nosection}{%
\fancyhead[CE]{}
\fancyhead[CO]{}
\fancyfoot[CE,CO]{\thepage}
}
\DeclareMathOperator{\R}{\mathbb{R}}

\renewcommand{\tilde}{\widetilde}

\newcommand\thankssymb[1]{\textsuperscript{{#1}}}

\newcommand{\dd}{d}

\newcommand{\p}{\partial}

\newcommand{\nm}{\noalign{\smallskip}}
\newcommand{\ds}{\displaystyle}

\usepackage{fancyhdr}

\fancypagestyle{plain}{%
\fancyhead[C]{}
\fancyfoot[C]{}

}
\fancypagestyle{nosection}{%
\fancyhead[CE]{}
\fancyhead[CO]{}
\fancyfoot[CE,CO]{\thepage}
}
\usepackage{amssymb}
\usepackage[protrusion=true,expansion=true]{microtype}	
\usepackage{amsmath,amsfonts,amsthm} % Math packages
\usepackage[pdftex]{graphicx}	
\usepackage{url}
\usepackage[title,titletoc,page]{appendix} 
\usepackage{listings} 
\usepackage{color}
\usepackage{comment}
\usepackage{caption}
\usepackage{subcaption}
\usepackage{algorithm}
\usepackage{fancyhdr}
\renewcommand{\epsilon}{\varepsilon}

\renewcommand{\tilde}{\widetilde}

\usepackage{tcolorbox}
\usetikzlibrary{tikzmark,calc}

\usepackage{enumerate}
\usepackage{upgreek}